\documentclass[article]{amsart}
\usepackage{cases}
\usepackage{amsmath, amsfonts, pifont, amssymb}
\usepackage{verbatim}
\usepackage[bookmarks=true]{hyperref}
\usepackage{geometry}
\newtheorem{theorem}{Theorem}[section]
\newtheorem{lemma}[theorem]{Lemma}
\newtheorem{proposition}[theorem]{Proposition}
\newtheorem{corollary}[theorem]{Corollary}

\newcommand{\beq}{\begin{equation}}
\newcommand{\eeq}{\end{equation}}
\newcommand{\beqq}{\begin{equation*}}
\newcommand{\eeqq}{\end{equation*}}

\theoremstyle{definition}

\theoremstyle{remark}
\newtheorem{remark}[theorem]{Remark}

\numberwithin{equation}{section}

\begin{document}

\address{Zehua Zhao
\newline \indent Department of Mathematics and Statistics, Beijing Institute of Technology, Beijing, China.
\newline \indent MIIT Key Laboratory of Mathematical Theory and Computation in Information Security, Beijing, China.}
\email{zzh@bit.edu.cn}

\title[Many body Schr\"odinger equation on waveguide manifolds]{On Strichartz estimate for many body Schr\"odinger equations in the waveguide setting}
\author{Zehua Zhao}
\maketitle

\setcounter{tocdepth}{1}
\tableofcontents

\begin{abstract}
In this short paper, we prove Strichartz estimates for N-body Schr\"odinger
equations in the waveguide manifold setting (i.e. on semiperiodic spaces $\mathbb{R}^m\times \mathbb{T}^n$ where $m\geq 3$), provided that interaction potentials are small enough (depending on the number of the particles and the universal constants, not on the initial data). The proof combines both the ideas of Tzvetkov-Visciglia \cite{TV1} and Hong \cite{hong2017strichartz}. As an immediate application, the scattering asymptotics for this model is also obtained. This result extends Hong \cite{hong2017strichartz} to the waveguide case. 
\end{abstract}
\bigskip

\noindent \textbf{Keywords}: Strichartz estimate, many body Schr\"odinger equations, scattering, waveguide manifolds
\bigskip

\noindent \textbf{Mathematics Subject Classification (2020)} Primary: 35Q55; Secondary: 35R01, 37K06, 37L50.

\section{introduction}
\subsection{Background and Motivations}
Let $d=m+n$, $m \geq 3$, $n \geq 1$ and $N \geq 1$. We consider the many body Schr\"odinger equations in the waveguide setting as follows,
\begin{equation}\label{maineq}
(i \partial_t+H_N)u(t,x_1,...x_N)=0, \quad u(0,x_1,...x_N)=u_0(x_1,...x_N) \in L^2_{x_1,...x_N} ,  
\end{equation}
where $H_N=\Delta_x-V_N=\sum_{\alpha=1}^{N} \Delta_{x_{\alpha}}-\sum_{1\leq \alpha<\beta \leq N}V(x_{\alpha}-x_{\beta})$, \\
and $\alpha$-th particle $x_{\alpha} \in \mathbb{R}^{m}\times \mathbb{T}^{n}$ for any $\alpha \in [1,...,N]$.

From physical explanations, $N \geq 1$ indicates the number of particles in a quantum system (which is often very large) and the interacting potentials of form $V(x_{\alpha}-x_{\beta})$ indicates the interactions of any two particles, which depends on their relative distance. Moreover, the product spaces of form $ \mathbb{R}^{m}\times \mathbb{T}^{n}$ is known as semi-periodic space or waveguide manifold. $d=m+n$ is the whole dimension while $m$ is the dimension for the Euclidean component and $n$ is the dimension for the tori component. 

When $N=1$, initial value problem\eqref{maineq} is exactly the standard nonlinear Schr\"odinger equation (NLS) with a potential, which has been well studied (in the Euclidean case, i.e. replacing $ \mathbb{R}^{m}\times \mathbb{T}^{n}$ by $ \mathbb{R}^{d}$). It is also known as `the one-body case' and the research on the decay properties has a long history (see the Introduction of \cite{hong2017strichartz}, the survey \cite{schlag2005dispersive} and the references therein). In this paper, we mainly concern the general case ($N\geq 1$ can be arbitrarily large), i.e. the many body Schr\"odinger case, which will involve some new difficulties than the single body case such as the issue of interacting potentials.  

The purpose of this paper is to investigate time decay properties of solutions to
the $N$-body Schr\"odinger equation \eqref{maineq} in the waveguide setting. In particular, we discuss the Strichartz-type estimate and the scattering behavior for \eqref{maineq}. We note that the Euclidean case of \eqref{maineq} has been studied in \cite{hong2017strichartz}. (See also \cite{chong2021global} for a recently result which deals with the two body case via the scheme of \cite{keel1998endpoint}.) We intend to generalize \cite{hong2017strichartz} to the waveguide case. That is one main motivation of this paper. Another motivation is the recent developments for the topic: `Long time behavior for NLS on waveguides' so the author is interested in combining both of `waveguides' and `many body Schr\"odinger equations' together, i.e. studying the estimates and the long time behavior for many body Schr\"odinger equations on waveguides. We will briefly mention the background  for `NLS on waveguides' in the next paragraph.

Waveguide manifolds of form $ \mathbb{R}^{m}\times \mathbb{T}^{n}$ are of particular interest in nonlinear optics of telecommunications. Generally, well-posedness theory and long time behavior of NLS are hot topics in the area of dispersive equations and have been studied widely in recent decades. Naturally, the Euclidean case is first treated and the theory, at least in the defocusing setting, has been well established. We refer to \cite{Iteam1,BD3,KM1} for some typical  Euclidean results. Moreover, we refer to \cite{CZZ,CGZ,HP,HTT1,HTT2,IPT3,IPRT3,KV1,yang2023scattering,YYZ,Z1,Z2,ZhaoZheng} with regard to the torus and waveguide settings. (See also \cite{LYYZ,sire2022scattering,yu2021global} for other dispersive equations on waveguides.) One may roughly think of the waveguide case as the ``intermediate point" between the Euclidean case and the torus case since the waveguide manifold is a product of Euclidean spaces and the tori. The techniques used in Euclidean and torus settings are frequently combined and applied to the waveguides problems. At last, we refer to \cite{Cbook,dodson2019defocusing,Taobook} for some classical textbooks on the study of NLS.

Since the current paper concerns the estimates and the PDE-level of \eqref{maineq} rather than the mathematical physics level, we will not mention too much for the background of the many body problems/equations from physical perspectives. We refer to the Introductions of \cite{chen2017global,chong2022global,chong2021global,chong2019dynamical,grillakis2013pair,sigal1987n} and the references therein for more information.

To the authors' best knowledge, the current paper is the first result towards understanding long time dynamics for the many body Schr\"odinger equations within the context of waveguides.

 As last, we note that, as in \cite{hong2017strichartz}, we need to assume some smallness for the potential $V$ and this smallness does not depend on the initial data (only depends on the particle number $N$ and the universal constant). 
\subsection{The statement of main results}
Now we are ready to state the two main results of this paper. We start with the Strichartz estimate as follows since the other one is an application of it.
\begin{theorem}[Strichartz estimate]\label{thm1}
Let $m \geq 3$, $n \geq 1$ and $1<p<2$. We also recall $N$ and $H_N$ in \eqref{maineq}. There exists a small number $\epsilon$ such that if $\|V\|_{L_y^{\frac{d}{2},\infty}L^2_z} \leq \frac{\epsilon}{N^2}$, then 
\begin{equation}
\|\mathbf{1}_{[0,+\infty)} e^{-it H_N} u_0 \|_{V^p_{\Delta_x}} \lesssim \|u_0\|_{L^2_x}.
\end{equation}
\end{theorem}
\begin{remark}
Here $V^p_{\Delta_x}$-norm (known as variation spaces) is introduced by Koch-Tataru \cite{koch2007priori} (see Section \ref{pre} for discussions). See 
 also \cite{HTT1,HTT2,IPT3,IPRT3} for more information and some other applications.  
\end{remark}
In viewing of the properties of $V^p_{\Delta_x}$-type spaces, Theorem \ref{thm1} directly implies
\begin{corollary}\label{cor1}
Let $m \geq 3$ and $n \geq 1$. There exists a small number $\epsilon$ such that if $\|V\|_{L_y^{\frac{d}{2},\infty}L^2_z} \leq \frac{\epsilon}{N^2}$, then for any
$m
$-dimensional admissible pair $(q,r)$ and $1 \leq  \alpha \leq N$, we have
\begin{equation}\label{Stri1}
\| e^{-it H_N} u_0 \|_{L^q_tL^r_{y_{\alpha}}L^2_{z_{\alpha}}L^2_{\hat{x}_{\alpha}}} \lesssim \|u_0\|_{L^2_x},
\end{equation}
where $\hat{x}_{\alpha}$ is the $N-1$ spatial variables except the $\alpha$-th variable $x_{\alpha}$, i.e.,
\begin{equation}
   \hat{x}_{\alpha}=(x_1,...x_{\alpha-1},x_{\alpha+1}...,x_N) \in (\mathbb{R}^m\times \mathbb{T}^n)^{(N-1)},
\end{equation}
and $x_{\alpha}$ is the $\alpha$-th variable with Euclidean component $y_{\alpha}$ and tori component $z_{\alpha}$, i.e.,
\begin{equation}
   x_{\alpha}=(y_{\alpha},z_{\alpha})\in \mathbb{R}^m\times \mathbb{T}^n.
\end{equation}
Moreover, for any $mN$-dimensional admissible pair $(q,r)$, we have
\begin{equation}\label{Stri2}
\| e^{-it H_N} u_0 \|_{L^q_tL^r_{y}L^2_{z}} \lesssim \|u_0\|_{L^2_x},
\end{equation}
where $y$ is for the whole Euclidean component ($mN$-dimensional) and $z$ is for the whole tori component $nN$-dimensional.
\end{corollary}
\begin{remark}
 See Theorem 1.1 and 
 Theorem 1.2 in \cite{hong2017strichartz} for the Euclidean case. We will give the proof for Corollary \ref{cor1} after the proof of Theorem \ref{thm1} in the end of Section \ref{3}.   
\end{remark}

\begin{remark}
As shown above, the formulation of the Strichartz estimates for \eqref{maineq} combines both the ideas of \cite{TV1} and \cite{hong2017strichartz}. As in \cite{TV1} (see also \cite{TV2}), we fix the tori component by using $L^2$-norm. (In other words, one decomposes the function along the tori direction and derive the Strichartz estimate using the dispersion from the Euclidean direction.) As in \cite{hong2017strichartz}, we fix other particles by only considering the dispersion of one certain particle. Thus, we consider the dispersion of the Euclidean component of one particle; fixing other particles and the tori component of this particle by using $L^2$-norm.
\end{remark}

As a direct application of Theorem \ref{thm1}, we show the scattering behavior for an $N$-body Schr\"odinger operator with
rough small interactions in the following sense,
\begin{theorem}[Scattering]\label{thm2}
Let $m \geq 3$, $n \geq 1$ and $1<p<2$. let $\epsilon$ be a small constant given in
Theorem \ref{thm1}. If $\|V\|_{L_y^{\frac{m}{2},\infty}L^2_z} \leq \frac{\epsilon}{N^2}$, then for each $u_0 \in  L^2_x$, there exist scattering states
$u_{\pm}$ such that
\begin{equation}
\lim_{t\rightarrow \pm \infty} \big\| e^{-itH_N}u_0-e^{it\Delta_x}u_{\pm} \big\|_{L^2_x}=0.    
\end{equation}
\end{theorem}
\begin{remark}
We note that for the tori case of \eqref{maineq}, the scattering behavior is not expected due to the lack of dispersion, though a Strichartz estimate can still be possibly obtained with suitable modifications. We leave it for interested readers. 
\end{remark}
\begin{remark}
For the above results, the dimension of the tori component $n \geq 1$ does not matter. (When $n=0$, it is exactly the Euclidean case \cite{hong2017strichartz}). However, if one considers the long time dynamics of a nonlinear problem on waveguide manifolds, the dimension of the tori component often matters a lot. In general, the difficulty of the critical NLS problem on waveguide manifolds increases if the whole dimension is increased or if the Euclidean component is decreased. See the Introductions in \cite{IPT3,IPRT3} for more information.
\end{remark}

\begin{remark}
To be more general, the
tori component $\mathbb{T}^n$ in \eqref{maineq} can be generalized to a compact Riemannian manifold $\mathcal{M}$ such that Theorem \ref{thm1}, Corollary \ref{cor1} and Theorem \ref{thm2} still hold.
\end{remark}
Next, we briefly introduce the main strategy of the proofs for Theorem \ref{thm1}, Corollary \ref{cor1} and Theorem \ref{thm2}. In fact, the proofs for Corollary \ref{cor1} and Theorem \ref{thm2} are standard and less complicated. Corollary \ref{cor1} follows from Theorem \ref{thm1} (see Section \ref{3}) according to the transfer principle of the function space $V^p_{\Delta_x}$. Theorem \ref{thm2} also follows from Theorem \ref{thm1} (see Section \ref{4}), together with some other basic estimates like in \cite{hong2017strichartz}. Thus we will focus on the proof of Theorem \ref{thm1} as follows.

The proof of Theorem \ref{thm1} (Strichartz estimate) is based on the properties of function space $V^p_{\Delta_x}$ and a perturbation method (see Section \ref{3}). The main idea is: one establishes nonlinear estimate for one arbitrary interacting potential (treating it as a perturbation) and then sum them up. The key estimate is Proposition \ref{key} which deals with one arbitrary interacting potential by regarding it as a forcing term. With the help of it, one can handle all of the interacting potentials by treating them as perturbations. Eventually, according to the smallness assumption, one can use perturbation method to show the Strichartz estimate as desired. Compared with the single potential case ($N=1$), the interacting potentials (involves rotations) cause difficulties thus the `rotation flexible' function space $V^p_{\Delta_x}$ is needed; compared with the Euclidean analogue (\cite{hong2017strichartz}), the new difference is the appearance of the tori component.  
\subsection{Structure of this paper}
The rest of the article is organized as follows. In Section \ref{pre}, we discuss function spaces and some estimates for this model; in Section \ref{3}, we give the proof for Theorem \ref{thm1} (Strichartz estimate); in Section \ref{4}, we give the proof for Theorem \ref{thm2} (scattering asymptotics); in Section \ref{5}, we give a few further remarks on this research line.
\subsection{Notations}
We write $A \lesssim B$ to say that there is a constant $C$ such that $A\leq CB$. We use $A \simeq B$ when $A \lesssim B \lesssim A $. Particularly, we write $A \lesssim_u B$ to express that $A\leq C(u)B$ for some constant $C(u)$ depending on $u$. We use $C$ for universal constants and $N$ for the number of particles.

We say that the pair $(p,q)$ is $d$-(Strichartz) admissible if
\begin{equation}
    \frac{2}{p}+\frac{d}{q}=\frac{d}{2},\quad 2 \leq p,q \leq \infty \quad (p,q,d)\neq (2,\infty,2).
\end{equation}
Throughout this paper, we regularly refer to the spacetime norms
\begin{equation}
    \|u\|_{L^p_tL^q_z(I_t \times \mathbb{R}^m\times \mathbb{T}^n)}=\left(\int_{I_t}\left(\int_{\mathbb{R}^m\times \mathbb{T}^n} |u(t,z)|^q dz \right)^{\frac{p}{q}} dt\right)^{\frac{1}{p}}.
\end{equation}
Similarly we can define the composition of three $L^p$-type norms like $L^p_tL^q_xL^2_y$. As in Theorems \ref{thm1}, \ref{thm2} and Corollary \ref{cor1}, we use $L^{r,s}$ for the Lorentz norm (see \cite{bergh2012interpolation}). One can define the composition of norms in a similar way.

As stated in the above Theorems, in general, we \textbf{refer to} $x$ for the whole spatial variable; $y$ for the whole Euclidean spatial variable; $z$ for the whole tori spatial variable; $x_{\alpha}$ for the $\alpha$-th spatial variable; $y_{\alpha}$ for the $\alpha$-th Euclidean spatial variable; $z_{\alpha}$ for the $\alpha$-th tori spatial variable for convenience.

Similar to the Euclidean case, function spaces such as $V^p_{\Delta}$ are also tightly involved. we will discuss them in Section \ref{pre}. (See also \cite{hong2017strichartz}.) 

To deal with the interacting potentials, we define the rotation operator $\mathcal{R}_{\alpha \beta}$ by
\begin{equation}
\mathcal{R}_{\alpha \beta}(f(x_1,...x_{\alpha-1},\frac{x_{\alpha}-x_{\beta}}{\sqrt{2}},x_{\alpha+1}...x_{\beta-1},\frac{x_{\alpha}+x_{\beta}}{\sqrt{2}},x_{\beta+1},...x_N))=f(x_1,...x_N).    
\end{equation}
\subsection*{Acknowledgment} The author was supported by the NSF grant of China (No. 12101046, 12271032) and the Beijing Institute of Technology Research Fund Program for Young Scholars. The author has learned many body Schr\"odinger model and related background during his postdoc career at University of Maryland (2019-2021). Thus he highly appreciates Prof. M. Grillakis, Prof. M. Machedon and their group (Dr. J. Chong and Dr. X. Huang) for related discussions, especially the paper of Hong \cite{hong2017strichartz}.
\section{Preliminaries}\label{pre}
In this section, we discuss function spaces and some estimates for the model \eqref{maineq}. See Section 2 to Section 4 in \cite{hong2017strichartz} for the Euclidean analogue.

First, similar to the Euclidean case, one can easily show: if the potential $V$ is small enough, then the Strichartz estimate for operator $e^{it(\Delta_x-V)}$ also holds for the waveguide case as follows.
\begin{lemma}\label{single}
Let $m \geq 3$, $n \geq 1$, and let $c_0$ be the
implicit constant given in Proposition \ref{Strichartzwaveguide}. If $\|V\|_{L_y^{\frac{d}{2},\infty}L^2_z} < \frac{1}{c_0}$, then
\begin{equation}
\|e^{it(\Delta_x-V)} u_0\|_{L_t^q L_y^r L^2_z(\mathbb{R}\times \mathbb{R}^m \times \mathbb{T}^n)} \leq \frac{c_0}{1-c_0 \|V\|_{L_y^{\frac{d}{2},\infty}L^2_z}} \|u_0\|_{L^2_{y,z}( \mathbb{R}^m \times \mathbb{T}^n)},    
\end{equation}
for all $m$-admissible pair $(q,r)$.
\end{lemma}
As in the Euclidean case, to finish the proof of Lemma \ref{single}, recall the Strichartz estimates in the waveguide setting as follows (see Proposition 2.1 in \cite{TV1}, in fact, this result is more general since it concerns the compact Riemannian manifold case).
\begin{proposition}\label{Strichartzwaveguide}
For every $n \geq 1$ and for every compact Riemannian manifold $M^k_y$, one considers functions $f(x,y), F(x,y)$ on $\mathbb{R}^n \times M^k_y$, then the
following estimate holds:  
\begin{equation}
\|e^{it\Delta_{x,y}}f\|_{L^p_t L^q_x L^2_y}+\big\|\int_0^t e^{i(t-s)\Delta_{x,y}}F(s,x,y)ds\big\|_{L^p_t L^q_x L^2_y}
\lesssim \|f\|_{L^2_{x,y}}+\|F\|_{L^{\tilde{p}^{'}}_{t}L^{\tilde{q}^{'}}_xL_y^2}, 
\end{equation}
where $(p,q)$ and $(\tilde{p},\tilde{q})$ are Strichartz admissible pairs.
\end{proposition}

\begin{proof}[Proof of Lemma \ref{single}]
 Since the potential is small, the proof is purely perturbative. One can just use waveguide Strichartz estimate Proposition 
 \ref{Strichartzwaveguide} to treat the potential as a perturbation term provided the potential is small (the Duhamel's formula
and the H\"older inequality are also used). Thus we omit the proof. See Theorem 2.1 in \cite{hong2017strichartz} for the Euclidean analogue.
\end{proof}
Next, we discuss Strichartz estimates with frozen spatial variables. (See Proposition 1 in \cite{hong2017strichartz} and Theorem 3.1 in \cite{chen2017global} for the Euclidean analogue.) The difference is that: now we fix both of the tori component of a certain particle and the other particles by using $L^2$-norm. In other words, the `frozen spatial variables' are the tori component and the other particles. Standard dispersive estimate and an important lemma in \cite{keel1998endpoint} which `lifts' dispersive estimates to Strichartz estimates are used.
\begin{proposition}\label{Stri}
Let $m \geq 3$ and $n \geq 1$. Then for any
$m
$-dimensional admissible pair $(q,r),(\tilde{q},\tilde{r})$ and $1 \leq  \alpha \leq N$, we have   
\begin{equation}
\| e^{it \Delta_x} u_0 \|_{L^q_tL^{r,2}_{y_{\alpha}}L^2_{z_{\alpha}}L^2_{\hat{x}_{\alpha}}} \lesssim \|u_0\|_{L^2_x},    
\end{equation}
\begin{equation}
\| \int_{\mathbb{R}} e^{-is\Delta_x} F(s)ds \|_{L^2_x} \lesssim \|F\|_{L^{\tilde{q}^{'}}_tL^{\tilde{r}^{'},2}_{y_{\alpha}}L^2_{z_{\alpha}}L^2_{\hat{x}_{\alpha}}},    
\end{equation}
and
\begin{equation}
\| \int_{0}^{t} e^{-i(t-s)\Delta_x} F(s)ds \|_{L^q_tL^{r,2}_{y_{\alpha}}L^2_{z_{\alpha}}L^2_{\hat{x}_{\alpha}}} \lesssim \|F\|_{L^{\tilde{q}^{'}}_tL^{\tilde{r}^{'},2}_{y_{\alpha}}L^2_{z_{\alpha}}L^2_{\hat{x}_{\alpha}}},    
\end{equation}
where 
\begin{equation}
   \hat{x}_{\alpha}=(x_1,...x_{\alpha-1},x_{\alpha+1}...,x_N) \in \mathbb{R}^{d(N-1)},
\end{equation}
and $x_{\alpha}$ is the $\alpha$-th variable with Euclidean component $y_{\alpha}$ and tori component $z_{\alpha}$, i.e.,
\begin{equation}
   x_{\alpha}=(y_{\alpha},z_{\alpha})\in \mathbb{R}^m\times \mathbb{T}^n.
\end{equation}
\end{proposition}
\begin{proof}
We consider a complex-valued function $f(x): \mathbb{R}^{dN}_x \rightarrow \mathbb{C}$
in $L^{r,2}_{y_{\alpha}}L^2_{z_{\alpha}}L^2_{\hat{x}_{\alpha}}$ with the function-valued function $f(y_{\alpha};z_{\alpha},\hat{x}_{\alpha})$ in $L^{r,2}_{y_{\alpha}}$. We note that $r \geq 2$. Using unitarity property,
\begin{equation}
\| e^{it \Delta_x} u_0 \|_{L^{r}_{y_{\alpha}}L^2_{z_{\alpha}}L^2_{\hat{x}_{\alpha}}}=  \| e^{it \Delta_{y_{\alpha}}} u_0 \|_{L^{r}_{y_{\alpha}}L^2_{z_{\alpha}}L^2_{\hat{x}_{\alpha}}}=\| e^{it \Delta_{y_{\alpha}}} u_0 \|_{L^2_{z_{\alpha}}L^2_{\hat{x}_{\alpha}}L^{r}_{y_{\alpha}}}.  
\end{equation}
Then, by the standard dispersive estimate (for $y_{\alpha}$-direction which is $m$-dimensional)
\begin{equation}
\| e^{it \Delta_{y_{\alpha}}} u_0 \|_{L^{r}_{y_{\alpha}}} \lesssim \frac{1}{|t|^{m(\frac{1}{2}-\frac{1}{r})}} \|f\|_{L^{r^{'}}_{y_{\alpha}}},    
\end{equation}
we obtain (the Minkowski allows one change the order of norms)
\begin{equation}
 \| e^{it \Delta_x} u_0 \|_{L^{r}_{y_{\alpha}}L^2_{z_{\alpha}}L^2_{\hat{x}_{\alpha}}} \lesssim  \frac{1}{|t|^{m(\frac{1}{2}-\frac{1}{r})}} \| u_0 \|_{L^2_{z_{\alpha}}L^2_{\hat{x}_{\alpha}}L^{r^{'}}_{y_{\alpha}}} \lesssim  \frac{1}{|t|^{m(\frac{1}{2}-\frac{1}{r})}} \|  u_0 \|_{L^{r^{'}}_{y_{\alpha}}L^2_{z_{\alpha}}L^2_{\hat{x}_{\alpha}}}.
\end{equation}
The proposition follows from Theorem 10.1 in \cite{keel1998endpoint}.
\end{proof}
Now we briefly discuss the function spaces and corresponding estimates. They will be essentially used in the following two sections. As mentioned in the end of Section 3.3 in \cite{hong2017strichartz},  Strichartz estimates with frozen spatial variables are still not sufficient to complete the proof of Theorem \ref{thm1} (Strichartz estimate) because of the interacting potentials. That is why a space-time norm that plays the
role of the rotated space-time norm is needed. This part is almost the same as Section 4.1 in \cite{hong2017strichartz} with natural modifications. We also refer to \cite{HTT1,HTT2,koch2007priori} for more details. 

We note that the definitions and properties in Subsection 4.1. of \cite{hong2017strichartz} are general enough which can be applied for our model in the waveguide setting naturally. They construct function spaces with nice properties for a separable Hilbert space $H$ and self-adjoint operator $S$. In this paper, we can just choose $H$ to be $L^2_x$
and $S$ to be $\Delta_{x}$ in the waveguide setting, where $x=(x_1,...,x_N)$ and $x_{\alpha}\in \mathbb{R}^m \times \mathbb{T}^n$ for $\alpha \in \{1,...,N\}$ as in \eqref{maineq}. Then the definitions and associated properties for our case will hold as well. Thus we refer to Subsection 4.1. of \cite{hong2017strichartz} for the function spaces and corresponding estimates/properties. For instance, we will use the following property of $V^p_{\Delta}$-space. (It follows from the definition. See Proposition 2 in \cite{hong2017strichartz}.)
\begin{equation}
\|\mathbf{1}_{[0,\infty)} e^{it\Delta_x} u_0\|_{V^p_{\Delta_x}}= \|u_0\|_{L^2_x}.   
\end{equation}
Moreover, the duality, the inclusion properties and the transference principle of $V^p_{\Delta}$-space are also often used. (See Subsection 4.1. of \cite{hong2017strichartz}) 
\section{The proof of Theorem \ref{thm1}}\label{3}
In this section, we discuss the proof of Theorem \ref{thm1} (Strichartz estimate). Corollary \ref{cor1} will also be obtained using the properties of function space $V^p_{\Delta}$. Like in \cite{hong2017strichartz}, we handle the potential terms by treating them as perturbations. The key estimate is as follows,
\begin{proposition}\label{key}
Let $m \geq 3$, $n \geq 1$ and $1<p<2$. Consider $u$ in the waveguide setting as in Theorem \ref{thm1}. Then, we have
\begin{equation}
\big\| \mathbf{1}_{[0,+\infty)} \int_0^t e^{i(t-s)\Delta_x}(V(x_{\alpha}-x_{\beta})u(s))  ds  \big\|_{V^p_{\Delta}} \leq C \|V\|_{L_y^{\frac{m}{2},\infty}L^2_z} \|u\|_{V^p_{\Delta}} ,
\end{equation}
where $C$ is for the universal constant.
\end{proposition}
\begin{remark}
  Proposition \ref{key} indicates that one can regard the potential terms as perturbations. As we can see from the proof below, it suffices to consider one arbitrary interacting potential $V(x_{\alpha}-x_{\beta})$ since the $V^p_{\Delta}$-norm is rotation-flexible.
\end{remark}
\begin{remark}
  See Proposition 4 in \cite{hong2017strichartz} for the Euclidean analogue. The main new difference for the waveguide case is the appearance of the tori component. 
\end{remark}
\begin{proof}

For notational convenience, we denote
\begin{equation}
    w=\textbf{1}_{[0,\infty)} \int_0^t e^{-is\Delta_{x}}(F(s))ds,
\end{equation}
where $F=V(x_{\alpha}-x_{\beta})u(s)$ is treated as the forcing term (or say a perturbative term).

We will estimate $w$ by the duality argument. Since we only expect $w \in V^p_{-}$, not $w \in V^p$, we consider $\tilde{w}(t)=w(-t)$.

Similar to Proposition 4 of \cite{hong2017strichartz}, using duality, it suffices to show that
\begin{equation}
    \sum_{j=1}^{J}\langle a(t_{j-1}),\tilde{w}(j)-\tilde{w}(t_{j-1}) \rangle_{L^2_x}  \lesssim \|V\|_{L_y^{\frac{m}{2},\infty}L^2_z} \|u\|_{V^p_{\Delta}}
\end{equation}
for any fine partition of unity $t=\{t_j\}_{j=0}^J$ and any $U^{p^{'}}$-atom $a(t)=\sum_{k=1}^{K} \textbf{1}_{(s_{k-1},s_k)}\phi_{k-1}$. (We note that the $U^{p^{'}}$-space is the dual of the $V^p_{\Delta}$-space.)

Doing some standard simplifications as in Proposition 4 of \cite{hong2017strichartz} (expanding atoms $a$ in terms of $\phi_k$), one can get a simpler sum
\begin{equation}
    \sum_{k=1}^{K}\langle \phi_{k-1}, \tilde{w}(s_k)-\tilde{w}(s_{k-1})\rangle_{L^2_x}.
\end{equation}
We further write it as
\begin{align}
&\sum_{k=1}^{K}\langle \phi_{k-1}, \tilde{w}(s_k)-\tilde{w}(s_{k-1})\rangle_{L^2_x} \\
&=  -\sum_{k=1}^{K} \int_{-s_k}^{-s_{k-1}} \langle \phi_{k-1},e^{-is\Delta_x}(F(s)) \rangle_{L^2_x} ds \\
&=  -\sum_{k=1}^{K} \int_{-s_k}^{-s_{k-1}} \langle e^{is\Delta_x}\mathcal{R} \phi_{k-1},\mathcal{R}(F(s)) \rangle_{L^2_x} ds \\
&=  -\sum_{k=1}^{K} \int_{\mathbb{R}} \langle e^{is\Delta_x}\mathcal{R} \phi_{k-1}, \textbf{1}_{[-s_k,-s_{k-1}]} \mathcal{R}(F(s)) \rangle_{L^2_x} ds,
\end{align}
where $\mathcal{R}$ denotes any rotation operator. (It is just $\mathcal{R}_{\alpha \beta}$ for interacting potential $V(x_{\alpha}-x_{\beta})$.) We want to control it by $\|V\|_{L_y^{\frac{m}{2},\infty}L^2_z} \|u\|_{V^p_{\Delta}}$.

Then, applying the H\"older inequality, the Strichartz estimate (Proposition \ref{Stri}) and the transference principle (Lemma \ref{trans}), we estimate it by
\begin{align}
& \sum_{k=1}^{K}\langle \phi_{k-1}, \tilde{w}(s_k)-\tilde{w}(s_{k-1})\rangle_{L^2_x} \\
&\lesssim \sum_{k=1}^{K} \|e^{it\Delta}\mathcal{R} \phi_{k-1}\|_{L^2_tL^{\frac{2m}{m-2},2}_{y_{\alpha}}L^2_{z_{\alpha}}L^2_{\hat{x}_{\alpha}}}\|\textbf{1}_{[-s_k,-s_{k-1}]} \mathcal{R}(F(s))\|_{L^2_tL^{\frac{2m}{m+2},2}_{y_{\alpha}}L^2_{z_{\alpha}}L^2_{\hat{x}_{\alpha}}} \\
&\lesssim \sum_{k=1}^{K} \|\phi_{k-1}\|_{L^2_x}  \|V\|_{L_y^{\frac{m}{2},\infty}L^2_z} \|\textbf{1}_{[-s_k,-s_{k-1}]} \mathcal{R}(u)\|_{L^2_tL^{\frac{2m}{m-2},2}_{y_{\alpha}}L^2_{z_{\alpha}}L^2_{\hat{x}_{\alpha}}} \\
&\lesssim \sum_{k=1}^{K} \|\phi_{k-1}\|_{L^2_x}  \|V\|_{L_y^{\frac{m}{2},\infty}L^2_z} \|\textbf{1}_{[-s_k,-s_{k-1}]} (u)\|_{V^p_{\Delta_x}} \\
&\lesssim \|V\|_{L_y^{\frac{m}{2},\infty}L^2_z} \big\| \|\phi_{k-1}\|_{L^2_x} \big\|_{l^{p^{'}}} \cdot \big\| \|\textbf{1}_{[-s_k,-s_{k-1}]} (u)\|_{V^p_{\Delta_x}} \big\|_{l^{p}} \\
&\lesssim \|V\|_{L_y^{\frac{m}{2},\infty}L^2_z} \big\| \|\textbf{1}_{[-s_k,-s_{k-1}]} (u)\|_{V^p_{\Delta_x}} \big\|_{l^{p}}.
\end{align}
We note that we have used the inclusion property of discrete $L^p$ spaces (i.e. $l^{p}$-spaces). ($1<p<2$ implies $p^{'}>2$.)

To close the argument, now it remains to show that
\begin{equation}
\big\| \|\textbf{1}_{[-s_k,-s_{k-1}]} (u)\|_{V^p_{\Delta_x}} \big\|_{l^{p}}=\big\{ \sum_{k=1}^{K} \|\mathbf{1}_{[-s_k,-s_{k-1})} u\|^{p}_{V^p_{\Delta_x}} \big\}^{\frac{1}{p}} \leq \|u\|_{V^p_{\Delta_x}}.    
\end{equation}
This estimate follows exactly as the Euclidean case (using the definition of $V^p_{\Delta_x}$). There is no difference in the waveguide setting. Thus the proof of Proposition \ref{key} is complete.
\end{proof}
With the help of Proposition \ref{key}, we give the proof of Theorem \ref{thm1} as follows. We can now treat the potential terms as perturbations.
\begin{proof}
 Applying Proposition \ref{key} to the Duhamel formula for $u=e^{-itH_N} u_0$, we have,   
 \begin{equation}
   \|\mathbf{1}_{[0,+\infty)} u(t)\|_{V^p_{\Delta}}  \leq \|u_0\|_{L^2_x}+\frac{N(N-1)}{2}C \|V\|_{L_y^{\frac{d}{2},\infty}L^2_z} \|u\|_{V^p_{\Delta}}.
 \end{equation}
Theorem \ref{thm1} now follows from the smallness assumption of potential $V$. ($\frac{N(N-1)}{2}$ is the number of interacting potentials.)
\end{proof}
Corollary \ref{cor1} follows from Theorem \ref{thm1} in viewing of the following lemma:
\begin{lemma}[Transference principle]\label{trans}
Let $d \geq 1$, $1<p<2$, $q\geq 2$ and $X$ be a Banach space. If a function $u: \mathbb{R} \rightarrow X$ satisfies the bound 
\begin{equation}
\|e^{\it\Delta_x}u_0\|_{L^q_t X} \lesssim \|u_0\|_{L^2_x},   
\end{equation}
then
\begin{equation}
 \|u\|_{L^q_t X} \lesssim \|u\|_{V^p_{\Delta_x}}.    
\end{equation}
\end{lemma}
\begin{remark}
 We note that the Bourgain spaces $X^{s,b}$ (also known as Fourier restriction space) enjoy the similar transfer principle (see \cite{Taobook} for more info.). As summarized in \cite{hong2017strichartz}, the Strichartz estimates in the $V^p_{\Delta_x}$ sharpen the bounds in $X^{s,b}$ by $0+$ in that Strichartz estimates in the $X^{s,b}$ space do not cover the endpoint Strichartz estimates, while those in the $V^p_{\Delta_x}$-space do.   
\end{remark}
See Proposition 3 in \cite{hong2017strichartz} for the proof. As a direct consequence, it shows that the 
$V^p_{\Delta_x}$-norm dominates the two Strichartz-type space-time norms in Corollary \ref{cor1}. Thus, Corollary \ref{cor1} follows from Theorem \ref{thm1}.
\section{The proof of Theorem \ref{thm2}}\label{4}
Now we are ready to discuss the proof of Theorem \ref{thm2}, i.e. the scattering for \eqref{maineq}. Since we have established proper Strichartz-type estimate, the proof will follow similarly as in \cite{hong2017strichartz}. For the sake of completeness, we include it as follows.

Without loss of generality, we only consider for the positive time. It suffices to show
that
\begin{equation}
u_{+}=\lim_{t \rightarrow +\infty} e^{-it\Delta_x}e^{-itH_N} u_0    
\end{equation}
exists in $L^2_x$ as $t \rightarrow \infty$. Indeed, by the Duhamel formula
\begin{align}
    &\|e^{-it_2\Delta}e^{-it_2 H_N}u_0-e^{-it_1\Delta}e^{-it_1 H_N}u_0\|_{L^2_x} \\
    &\leq \sum_{1 \leq \alpha<\beta \leq N} \big\| \int_{t_1}^{t_2} e^{-is\Delta_x}((V(x_{\alpha}-x_{\beta}))e^{-is H_N} u_0)ds \big\|_{L^2_x}.
\end{align}
It suffices to consider one single potential. According to Theorem \ref{thm1} and Corollary \ref{cor1}, we have that
\begin{align}
&\big\| \int_{t_1}^{t_2} e^{-is\Delta_x}((V(x_{\alpha}-x_{\beta}))e^{-is H_N} u_0)ds \big\|_{L^2_x} \\
&=\big\| \mathcal{R}_{\alpha \beta} \int_{t_1}^{t_2} e^{-is\Delta_x}((V(x_{\alpha}-x_{\beta}))e^{-is H_N} u_0)ds \big\|_{L^2_x}\\
&=\big\| \int_{t_1}^{t_2} e^{-is\Delta_x}(V(\sqrt{2} x_{\alpha})(\mathcal{R}_{\alpha \beta}e^{-is H_N} u_0))ds \big\|_{L^2_x}\\
&\leq c_0 \big\| V(\sqrt{2} x_{\alpha})(\mathcal{R}_{\alpha \beta}e^{-is H_N} u_0) \big\|_{L^2_{t\in [t_1,t_2]}L^{\frac{2d}{d+2},2}_{y_{\alpha}}L^2_{z_{\alpha}}L^2_{\hat{x}_{\alpha}}}\\
&\leq \frac{c_0}{2} \|V\|_{L_z^{\frac{d}{2},\infty}L^2_z} \big\| (\mathcal{R}_{\alpha \beta}e^{-is H_N} u_0) \big\|_{L^2_{t\in [t_1,t_2]}L^{\frac{2d}{d-2},2}_{y_{\alpha}}L^2_{z_{\alpha}}L^2_{\hat{x}_{\alpha}}} \rightarrow 0
\end{align}
as $t_1,t_2 \rightarrow \infty$. Then we can see that the limit exists.
\section{Further remarks}\label{5}
In this section, we make a few more remarks for many body model \eqref{maineq} and Theorems \ref{thm1}, \ref{thm2} as follows.\vspace{1mm}

1. The main results in this paper and \cite{hong2017strichartz} are based on perturbative scheme which are tightly dependent on the smallness assumption of the potentials. One may consider removing the smallness assumption to prove Strichartz estimates like Theorem \ref{thm1} or Corollary \ref{cor1}. It may be hard to consider the general case thus the two body case may be a good model to start with. (See \cite{chong2021global} for the Euclidean case.)\vspace{1mm}

2. It is also interesting to consider many body equation with a nonlinearity $F(t,x_1,...,x_N)$ and study the long time behavior. There are few general theories and results regarding this topic, especially the scattering-type results. Also, it may be hard to consider the general case thus the two body case may still be a good model to start with. (The $\Lambda$-equation in the Hartree–Fock–Bogoliubov equations is an example for the two body case, though it is in a coupled system which makes it more complicated. See \cite{chong2022global,chong2021global}.) 

We also note that via the standard $T$-$T^{\ast}$ argument and the Christ-Kiselev lemma, one can obtain the inhomogeneous Strichartz analogue of Corollary \ref{cor1} (excluding the double endpoint case). (See \cite{Taobook}.) With the help of it, one may obtain the local well-posedness for \eqref{maineq} with a subcritical nonlinearity in the energy space. We leave it for interested readers.\vspace{1mm}  

3. One may also consider the tori analogue of \eqref{maineq} (replacing $ \mathbb{R}^{m}\times \mathbb{T}^{n}$ by $ \mathbb{T}^{d}$)) and obtain some estimates. The reason we consider the waveguide case is that we intend to study the scattering behavior, which is not expected for the tori case. \vspace{1mm}

4. The results in the current paper is only about the estimates and the PDE-level of many body Schr\"odinger equations. One may consider the many body Schr\"odinger equations in the tori setting or waveguide setting from the perspectives of mathematical physics. (See \cite{chong2019dynamical,grillakis2013pair,grillakis2017pair} for examples.)

 \bibliographystyle{amsplain}
\bibliography{manybody}
\bibliographystyle{plain}

\end{document}